\documentclass[12pt]{amsart}
\usepackage{amscd,amssymb,graphics,color,a4wide,hyperref,verbatim}

\usepackage{multicol, fullpage} 
\usepackage[usenames,dvipsnames]{xcolor} 
\usepackage{tikz, tikz-3dplot, pgfplots}
\usepackage{tkz-graph}
\usepackage{tikz-cd}
\usetikzlibrary[positioning,patterns] 
\usetikzlibrary{matrix,arrows,decorations.pathmorphing}

\usepackage{graphicx}
\usepackage{psfrag}
\usepackage{marvosym}
\usepackage{amsmath}
\usepackage{amsfonts}
\usepackage{amssymb}
\usepackage{amsthm}
\usepackage{mathrsfs}
\usepackage{enumitem}

\usepackage{ytableau}

\input xy
\xyoption{all}

\usepackage{calligra}
\usepackage{mathrsfs}
\DeclareMathAlphabet{\mathcalligra}{T1}{calligra}{m}{n}
\DeclareFontShape{T1}{calligra}{m}{n}{<->s*[1.5]callig15}{}

\newtheorem{theorem}{Theorem}[section]

\newtheorem{lemma}[theorem]{Lemma}

\newtheorem*{questionstar}{Question}

\theoremstyle{definition}
\newtheorem{definition}[theorem]{Definition}

\newtheorem{remark}[theorem]{Remark}
\newtheorem{theorem-definition}[theorem]{Theorem-Definition}

\numberwithin{equation}{section}

\newcommand{\LL} {\mathbb{L}}

\newcommand{\PP} {\mathbb{P}}

\newcommand{\RR} {\mathbb{R}}

\newcommand{\ZZ} {\mathbb{Z}}

\newcommand {\shA} {\mathcal{A}}

\newcommand {\shE} {\mathcal{E}}

\newcommand {\shH} {\mathcal{H}}

\newcommand {\shT} {\mathcal{T}}
\newcommand {\shP} {\mathcal{P}}

\newcommand {\sC} {\mathscr{C}}

\newcommand {\sE} {\mathscr{E}}

\newcommand {\sI} {\mathscr{I}}

\newcommand {\sL} {\mathscr{L}}

\newcommand {\sN} {\mathscr{N}}
\newcommand {\sO} {\mathscr{O}}

\newcommand {\sU} {\mathscr{U}}

\newcommand {\foR} {\mathfrak{R}}

\newcommand {\coh} {\operatorname{coh}}

\newcommand{\sExt}{\mathscr{E} \kern -1pt xt}

\newcommand {\Hom} {\operatorname{Hom}}
\newcommand {\sHom}{\mathscr{H}\kern-5pt\mathcalligra{om}}

\newcommand {\kk} {\Bbbk}

\newcommand {\Proj} {\operatorname{Proj}}

\newcommand {\rank} {\operatorname{rank}}

\newcommand {\Sing} {\operatorname{Sing}}

\newcommand {\Sym} {\operatorname{Sym}}

\newcommand{\sTor}{\mathscr{T} \kern -3pt or}

\newcommand {\Bl} {\operatorname{Bl}}

\newcommand {\tX} {Z}

\newcommand {\bL} {\mathbf{L}}
\newcommand {\bR} {\mathbf{R}}

\newcommand{\hpd}{{\natural}}

\title[]{Derived category of projectivization and generalized linear duality}

\author[Q.Y.\ JIANG]{Qingyuan Jiang}

\address{Institute for Advanced Study, Einstein Drive, Princeton, NJ 08540, USA}\email{jiangqy@ias.edu}

\begin{document}

\begin{abstract} In this note, we generalize the linear duality between vector subbundles (or equivalently quotient bundles) of dual vector bundles to coherent quotients $V \twoheadrightarrow \mathscr{L}$ considered in \cite{JL17}, in the framework of Kuznetsov's homological projective duality (HPD). 
As an application, we obtain a generalized version of the fundamental theorem of HPD for the $\mathbb{P}(\mathscr{L})$--sections and the respective dual sections of a given HPD pair.
\vspace{-2mm} 
\end{abstract}

\maketitle

\section{Introduction}
Let $S$ be a fixed scheme, which for simplicity we assume to be smooth over an algebraically closed field $\kk$ of characteristic zero, and $V$ be a vector bundle of rank $N \ge 2$ over $S$. Denote $V^\vee : = \sHom(V,\sO_S)$ the dual vector bundle. For a short exact sequence of vector bundles
	$$0 \to K \to V \to L \to 0,$$
over $S$, there is a dual short exact sequence of vector bundles
	\begin{equation} \label{eqn:dualL}
	0 \to L^\vee \to V^\vee \to K^\vee \to 0.
	\end{equation}
The {\em linear duality} refers to the duality of subbundles $\{K\subset V\} \leftrightarrow \{L^\vee \subset V^\vee\}$, or equivalently the quotient bundles $\{ V \twoheadrightarrow L\} \leftrightarrow \{ V^\vee \twoheadrightarrow K^\vee\}$. If we use Grothendieck convention $\PP(\sE): =  \Proj_S \Sym_{\sO_S}^\bullet \sE$ for a coherent sheaf $\sE$ on $S$, then linear duality equivalently refers to the reflexive relationship between all projective linear subbundles of $\PP(V)$ and $\PP(V^\vee)$:
	$$\{ \PP(L) \subset \PP(V) \} \quad \longleftrightarrow \quad \{ \PP(L)^\perp: = \PP(K^\vee) \subset \PP(V^\vee)\}.$$ 	
\begin{questionstar} What should be the {\em dual} of coherent quotient sheaves $\{V \twoheadrightarrow \sL\}$, or equivalently the subschemes $\{\PP(\sL) \subset \PP(V)\}$, where $\sL := {\rm coker}(K \to V)$ is {\em not} necessarily locally free?
\end{questionstar}

In this case we still have a short exact sequence of $\sO_S$-modules $0 \to K \to V \to \sL \to 0$; however the sequence (\ref{eqn:dualL}) is now replaced by a four-term exact sequence:
		\begin{equation} \label{eqn:dualsL}
		0 \to \sL^\vee \to V^\vee \to K^\vee \to \sExt^1_S(\sL,\sO_S) \to 0,
		\end{equation}
where $\sL^\vee: = \sHom_{S}(\sL,\sO_S)$, 
and $\sExt^1_S(\sL,\sO_S)$ 
is supported on the singular locus  
	${\rm Sing}(\sL) : = \{s \in S \mid \rank \sL(s) > \ell \} \subset S$
of $\sL$, where $\ell$ is the generic rank of $\sL$. 

In this note, we answer the above question in the framework homological projective duality:

\begin{theorem}[See Thm. \ref{thm:duality}] The homological projective dual (HPD) of $\PP(\sL) \subset \PP(V)$ is given by $\widetilde{\PP}(K^\vee) \to \PP(V^\vee)$, the blowing up of $\PP(K^\vee)$ along the $\PP(\sExt^1_S(\sL,\sO_S)) \subset \PP(K^\vee)$. 
\end{theorem}

The homological projective dual (HPD) of a Lefschetz variety $X \to \PP(V)$, introduced by Kuznetsov \cite{Kuz07HPD}, denoted by $Y = X^\hpd \to \PP(V^\vee)$, is a homological modification of the classical projective dual variety $X^\vee \subset \PP(V^\vee)$ of $X \to \PP(V)$, see \S \ref{sec:HPD} for precise definitions.

The HPD relation is {\em reflexive}: $(X^\hpd)^\hpd \simeq X$, see \cite{Kuz07HPD, JLX17}; And HPD {\em extends} the previously discussed linear duality  $\{ \PP(L) \subset \PP(V) \} \leftrightarrow \{ \PP(L)^\perp: = \PP(K^\vee) \subset \PP(V^\vee)\}$ between projective subbundles: $\PP(L)^\hpd = \PP(L)^\perp$, see \cite[Cor. 8.3]{Kuz07HPD}, \cite[Cor. 5.16]{JLX17}, but notice that our theorem (in the case when $\sL$ is locally free) also provides a different proof of this fact. Therefore, thanks to above theorem, it makes sense to denote:
	$$ \PP(\sL)^\perp : = \widetilde{\PP}(K^\vee) = \PP(\sL)^\hpd  \to \PP(V^\vee)$$ 
and regard it as the {\em dual} of $\PP(\sL) \subset \PP(V)$. The relation $ \PP(\sL) \leftrightarrow \PP(\sL)^\perp$ hence generalizes the usual linear duality. 

An immediate consequence of our theorem is the following generalization of the {fundamental theorem of HPD} from linear sections to the above generalized linear system $V \twoheadrightarrow \sL$. 

\begin{theorem}[Fundamental theorem of HPD for $V \twoheadrightarrow \sL$] \label{thm:HPD} Let $\shA$ be a $\PP(V^\vee)$-linear Lefschetz category of length $m$ with Lefschetz components $\shA_i$'s, and $\shA^\hpd$ be its HPD category, which is a $\PP(V)$-linear Lefschetz category of length $n$ with Lefschetz components $\shA_j^\hpd$'s. Then for $1 \le \ell \le N$, there are semiorthogonal decompositions
	\begin{align*}
\shA_{\PP(\sL) ^\perp}  &= \big\langle {}^{\rm prim}(\shA_{\PP(\sL) ^\perp}) , ~~  \shA_1^{\epsilon}(H),  \ldots, \shA_{\ell-1}^{\epsilon}((\ell-1)H), \\
			& \qquad \qquad \qquad \qquad   \langle \shA_{\ell}, \shA_{\ell}^{\epsilon} \rangle (\ell H), \ldots, \langle \shA_{m-1},  \shA_{m-1}^{\epsilon}\rangle ((m-1)H) \big\rangle, \\
\shA^{\hpd}_{\PP(\sL)}  & =  \big\langle  \shA^\hpd_{1-n}((\ell-n)H'), \ldots, \shA^\hpd_{-\ell}(-H') , ~~  (\shA^{\hpd}_{\PP(\sL)} )^{\rm prim}\big \rangle.
	\end{align*}
Furthermore, there is an equivalence of categories of the primitive components:
	$${}^{\rm prim}(\shA_{\PP(\sL) ^\perp}) \simeq (\shA^{\hpd}_{\PP(\sL)} )^{\rm prim}.$$
	\end{theorem} 
If $\sL$ is locally free, then the ``correction terms" $\shA_i^{\epsilon} = \emptyset$, and the theorem reduces to the usual fundamental theorem of HPD (see \cite{Kuz07HPD, JLX17, Ren, P18}).

If $\sL$ is not locally free, then there are nontrivial ``correction terms":
	$$\shA_{i}^{\epsilon} : = (\shA_{i}) |_{\PP(\sExt^1(\sL,\sO))}  = \shA_{i} \boxtimes_S D(\PP(\sExt^1(\sL,\sO))), \quad \text{for} \quad i = 1, \ldots, m-1,$$
supported on ${\rm Sing}(\sL) \subset S$. Our theorem shows that, after taking these corrections into consideration, the fundamental theorem of HPD still holds.

\subsection*{Acknowledgment} I am grateful for Richard Thomas, Andrei C\v{a}ld\v{a}raru, Mikhail Kapranov, Zak Turcinovic, Francesca Carocci for many helpful discussions, and my collaborators Conan Leung and Ying Xie for numerous discussions on HPD and our joint projects. I am supported by a grant from National Science Foundation (Grant No. DMS -- 1638352) and the Shiing-Shen Chern Membership Fund of IAS. 

\section{Preliminaries}
\subsection{Conventions} Let $S$ be a fixed scheme, for simplicity we assume to be smooth over an algebraically closed field $\kk$ of characteristic zero. All schemes considered in this paper will be $S$-schemes. Let $V$ be a fixed vector bundle of rank $N \ge 2$ over $S$, and $V^\vee$ be the dual vector bundle. We use Grothendieck convention $\PP(\sE): =  \Proj_S \Sym_{\sO_S}^\bullet \sE$ for a coherent sheaf $\sE$ on $S$. We use $D(X): = D^b_{\rm coh}(X)$ to denote the bounded derived categories of coherent sheaves on a scheme $X$.

Let $X$, $Y$ be $S$-schemes, and $f: X \to Y$ a proper $S$-morphism, then (whenever well defined) denote $\RR f_*$ and $\LL f^*$ the right and respectively left derived functors of usual pushforward $f_* \colon \coh X \to \coh Y$ and pullback $f^* \colon \coh Y \to \coh X$. Denote by $\otimes$, $\sHom(-,-)$ the tensor and sheaf (internal) $\Hom$ on $\coh X$, and $\otimes^\LL$ and $\RR \sHom(-,-)$ the derived functors. A \textit{Fourier-Mukai functor} is an exact functor between $D(X)$ and $D(Y)$ of the form 
	$$\Phi_{\shP}^{X \to Y} (-) = \Phi_{\shP} (-) : =  \RR \pi_{Y*} \,(\LL \pi_X^*\,(-) \otimes^\LL \shP): D(X) \to D(Y),$$
where $\shP \in D(X \times Y)$ is called the {\em Fourier-Mukai kernel}, and $\pi_{X}: X\times Y \to X$ and $\pi_Y: X \times Y \to Y$ are natural projections. 

\subsection{Generalities} The readers are referred to \cite{Huy, Cal, Kuz14sod} for basic notations and properties of derived categories of coherent sheaves, and semiorthogonal decompositions. 
	
A full triangulated subcategory $\shA$ of a triangulated category $\shT$ is called {\em admissible} if the inclusion functor $i = i_{\shA}: \shA\to \shT$ has both a right adjoint $i^!: \shT \to \shA$ and a left adjoint $i^*: \shT \to \shA$. If $\shA\subset \shT$ is admissible, then $\shA^\perp = \{ T \in \shT \mid \Hom(\shA,T) = 0\}$ and ${}^\perp \shA =\{ T \in \shT \mid \Hom(T, \shA) = 0\}$ are both admissible, and $\shT = \langle \shA^\perp, \shA \rangle =  \langle \shA, {}^\perp \shA\rangle$.

A {\em semiorthogonal decompositions} (SOD) for a triangulated category $\shT$, written as
		$$\shT = \langle \shA_1, \ldots, \shA_n \rangle.$$
 a sequence of admissible full triangulated subcategories $\shA_1, \ldots, \shA_{n}$, such that $(i)$ $\Hom (a_j ,a_i) = 0$ for all $a_i \in \shA_i$ and $a_j \in \shA_j$, $j > i$, and, $(ii)$ they generate the whole $D(X)$.
Starting with a semiorthogonal decomposition of $\shT$, one can obtain a whole collection of new decompositions by functors called {\em mutations}. The functor $\bL_\shA: = i_{\shA^\perp} i^*_{\shA^{\perp}}$  (resp. $\bR_{\shA} : =  i_{{}^\perp \shA} i^!_{ {}^\perp \shA}$) is called the {\em left (resp. right) mutation through $\shA$}.  For any $b \in \shT$, by there are exact triangles
		$$ i_\shA i^!_{\shA} (b) \to b \to \bL_{\shA} b \xrightarrow{[1]}{},\qquad  \bR_{\shA} b \to b \to  i_\shA i^*_{\shA} (b) \xrightarrow{[1]}{}.$$
$(\bL_{\shA})\,|_{\shA} = 0$ and $(\bR_{\shA})\,|_{\shA} = 0$ are the zero functors;  $(\bL_{\shA})\,|_{{}^\perp \shA} : {}^\perp \shA \to \shA^\perp$ and $(\bR_{\shA})\,|_{\shA^\perp } : \shA^\perp \to {}^\perp \shA $ are mutually inverse equivalences of categories. Staring with a semiorthogonal decomposition $\shT = \langle \shA_1, \ldots, \shA_{k-1}, \shA_k, \shA_{k+1}, \ldots, \shA_n \rangle$ of admissible subcategories, one can obtain other sods through mutations, for $k \in [1,n]$:
		\begin{align*}
		\shT&  = \langle \shA_1, \ldots,\shA_{k-2}, \bL_{\shA_{k-1}} (\shA_k), \shA_{k-1}, \shA_{k+1}, \ldots, \shA_n \rangle\\
		& =  \langle \shA_1, \ldots, \shA_{k-1}, \shA_{k+1}, \bR_{\shA_{k+1}} (\shA_k), \shA_{k+2}, \ldots, \shA_n \rangle
		\end{align*}
We refer the reader to \cite{BK, Kuz07HPD} for more about mutations.

For a  $S$-scheme $a \colon X \to S$ be a, $D(X)$ is naturally equipped with {\em $S$-linear structure}, given by $A \otimes a^* F$, for any $F \in D(S)$ and $A \in D(X)$. An admissible subcategory $\shA \subset D(X)$ is called {\em $S$-linear} if $A \otimes a^* F \in \shA$ for all $A \in \shA$ and $F \in D(S)$. Such an admissible subcategory $\shA$ will be referred as an {\em $S$-linear category}. An {\em $S$-linear} functor between $S$-linear categories is an exact functor functorially preserving $S$-linear structures. An $S$-linear SOD $D(X) = \langle \shA_1, \ldots, \shA_n \rangle$ for a $S$-scheme $X$ is a SOD such that all $\shA_i$'s are $S$-linear subcategories. See \cite{Kuz11Bas} for more about linear categories. Many geometric operations (projective bundles, blowing up, etc) can be performed on linear categories, see \cite{JL18Bl}. See also \cite{P18} for discussions in the Lurie's framework of stable $\infty$-categories.

\subsection{Generalized universal hyperplane section and Orlov's results} \label{sec:Cayley}
The references are \cite{RT15HPD,Orlov05}, see also \cite[\S 2.3]{JL17}, and \cite[\S 3.4]{JL18Bl} for noncommutative cases.

Let $\sE$ to be a locally free sheaf of rank $r$ on a regular scheme
$X$, and $s \in H^0(X,\sE)$ be a regular section. Denote $Z:=Z(s)$ the zero locus of the section $s$. Then through $H^0(X,\sE) = H^0(\PP(\sE), \sO_{\PP(\sE)}(1))$, the section $s$ corresponds to a section $f_s$ of $\sO_{\PP(\sE)}(1)$ on $\PP(\sE)$. The zero loci $\shH_s : = Z(f_s) \subset \PP(\sE)$ is called the {\em generalized universal hyperplane}, which comes with projection $\pi: \shH_s \to X$. The general fiber of this projection is a projective space $\PP^{r-2}$, and the fiber dimensions of $\pi$ jumps exactly over $Z$. If we denote $i: Z \hookrightarrow X$ the inclusion, then its normal sheaf is $\sN_{i} \simeq \sE|_Z$, and it is direct to see $\pi^{-1}(Z) = \PP(\sN_i)$. 

The above situation is called {\em Cayley's trick}. The situation is categorified by Orlov to obtain relationships between $D(Z)$ and $D(\shH_s)$ (see also \cite{JL17, JL18Bl}).

\begin{theorem}[Orlov, {\cite[Prop. 2.10]{Orlov05}}] \label{thm:HPDI} In the above situation, then the functors $\RR j_*\,p^*: D(Z) \to D(\shH_s)$ and $\LL \pi^*(-) \otimes \sO_{\shH_s}(k) : D(X) \to D(\shH_s)$ are fully faithful, where $k =1, \ldots, r-1$, $\sO_{\shH_s}(k) :=  \sO_{\PP(\sE)}(k)|_{\shH_s}$, and there is a semiorthogonal decomposition:
	\begin{align*}
	D(\shH_{s}) &= \langle \RR j_* \, p^* D(Z), ~~\LL \pi^* D(X) \otimes \sO_{\shH_s}(1), \ldots , \pi^* D(X) \otimes  \sO_{\shH_s}(r-1) \rangle, \\
		& =  \langle  \LL \pi^* D(X) \otimes  \sO_{\shH_s}(2-r),  \ldots,  \LL \pi^* D(X), ~~\RR j_* \, p^* D(Z) \rangle.
	\end{align*}
\end{theorem}

\subsection{Blowing up, and relation with Cayley's trick}
Suppose $Z$ is a codimension $r \ge 2$ locally complete intersection of a smooth variety $X$, the {\em blowing up of $Z$ along $X$} is $\pi: \Bl_Z X : = \PP(\sI_Z) \to X$, where $\sI_Z$ is the ideal sheaf of $Z$ inside $X$. The {\em exceptional divisor} is $i_E : E := \Bl_Z X \times_X Z \hookrightarrow \Bl_Z X$. Since $\sI_Z|_Z = \sN_{Z/X}^\vee$, therefore $E = \PP(\sN_{Z/X}^\vee)$. Denote $p\colon E \to Z$ be the projection. The following is due to Orlov \cite{Orlov92} (see also \cite{JL18Bl} for the case without smoothness condition on $Z$ and for the noncommutative case).

\begin{theorem}[Blowing up formula, Orlov \cite{Orlov92}] \label{thm:blow-up}
In the above situation, then the functors $\LL \pi^*: D(X) \to D(\Bl_Z X)$ and $ \RR i_{E*} \, \LL p^* (-) \otimes \sO(-kE) : D(Z) \to D(\Bl_Z X)$ are fully faithful, $k \in \ZZ$. Denote the image of the latter to be $D(Z)_k$, then 
	\begin{align*} D(\Bl_Z X)	& = \langle \LL \pi^* D(X), ~ D(Z)_0, D(Z)_1, \ldots, D(Z)_{r-2} \rangle; \\
				& = \langle D(Z)_{1-r}, \ldots, D(Z)_{-2}, D(Z)_{-1}, ~  \LL \pi^* D(X) \rangle. 			
	\end{align*}
\end{theorem}


\subsubsection{Relationship with Cayley's trick}\label{sec:BlvsCayley} There is a wonderful geometry relating blowing ups with Cayley's trick \cite{AW93}. In the situation of Cayley's trick (\S \ref{sec:Cayley}), if we pull back $\pi: \shH_s \to X$ along the blow-up $\beta: \Bl_Z X \to X$ of $X$ along $Z$, then the fiber product $\Bl_Z X \times_X \shH_s$ will have two irreducible components: one is $\PP(\sN_i) \times_Z \PP(\sN_i^\vee)$, the other is the {\em strict transform} of $\shH_s$ along the blow-up $\beta$, $\sU : = \overline{(\shH_s \,\backslash \PP(\sN_i)) \times_X \Bl_Z X} \subset \Bl_Z X \times_X \shH_s$. Then the projection $\pi_\sU: \sU \to \Bl_Z X$ will be a projective bundle of fiber $\PP^{r-2}$, and its restriction to $\PP(\sN_i^\vee)$ is nothing but the fiberwise incidence quadric $Q_Z \subset \PP(\sN_i) \times_Z \PP(\sN_i^\vee)$, which is defined fiberwisely over $z\in Z$ by incidence relation $\{(n,n^\vee)\in \PP(\sN_i|_z) \times \PP(\sN_i^\vee|_z)~|~\langle n,n^\vee\rangle = 0\}$. From blowing up closure lemma, $\sU$ is the blowing up of $\shH_s$ along $\PP(\sN_i)$:
	$$\sU = \Bl_{\PP(\sN_i)} \shH_s,\quad j_Q: Q_Z \hookrightarrow \sU ~~\text{is the exceptional divisor.}$$ 
Therefore we have a commutative diagram
\begin{equation*}
	\begin{tikzcd}[row sep= 2.6 em, column sep = 2.6 em]
	\sU \ar{d}[swap]{\pi_\sU} \ar{r}{\gamma} & \shH_s \ar{d}{\pi}  
	\\
	\Bl_Z X \ar{r}{\beta}         & X
	\end{tikzcd}	
\end{equation*}
relating the projection $\pi: \shH_s \to X$ from the universal hyperplane with the projection $\pi_\sU$ of a projective bundle, via the two blow-ups $\beta$ and $\gamma$. Notice the pullback $\widetilde{q}: \Bl_Z X \times_X \PP(\sE)  \to \Bl_Z X$ of projective bundle $q$ along $\beta$
is also projective bundle over $\Bl_Z X$, and also the divisor inclusion $\iota_\sU : \sU \hookrightarrow \Bl_Z X \times_X \PP(\sE)$ is defined by fiberwise quadric incidence relation (between $\Bl_Z X \subset \PP(\sE^\vee)$ and $\PP(\sE)$), i.e. $\sU$ is the {\em universal hyperplane} for $\Bl_Z X \subset \PP(\sE^\vee)$ over $X$ in the language of HPD \S \ref{sec:HPD}.

\begin{remark} Another way of understanding this picture (\cite{AW93}) is: $\Bl_Z X$ is the connected component of {\em Hilbert scheme} parametrizing the deformations of a general fiber $\PP^{r-2}$ inside $\shH_s$; $\sU$ is the {\em universal family}, therefore a projective bundle with fiber $\PP^{r-2}$ over $\Bl_Z X$. 
\end{remark}

\begin{lemma}[{\cite[Prop. 3.4]{CT15}}, {\cite[Lem. 2.9]{JL17}}] \label{lem:mut:Bl} In the situation of blowing up formula Thm. \ref{thm:blow-up}, for any $\shE^\bullet \in D(X)$, $k \in \ZZ$, we have the following equalities in $D(\Bl_Z X)$:
	\begin{align*} & \mathbf{L}_{D(Z)_k} (\LL \pi^* \shE^\bullet \otimes \sO(-(k+1)E)) = \LL \pi^* \shE^\bullet \otimes \sO(-kE), \\
	& \mathbf{R}_{D(Z)_k} (\LL \pi^* \shE^\bullet \otimes \sO(-kE)) =\LL \pi^* \shE^\bullet \otimes \sO(-(k+1)E).
	\end{align*}
\end{lemma}

\subsection{Lefschetz varieties and HPD} \label{sec:HPD} Lefschetz categories are the key ingredients for HPD theory. A variety $X \to \PP(V)$ is said to admit a {\em (right) Lefschetz decomposition} with respect to $\sO_{\PP(V)}(1)$ if there is a semiorthogonal decomposition of the form:
	\begin{equation*}\label{lef:X} 
	D(X)= \langle \shA_0, \shA_1(1), \ldots, \shA_{m-1}(m-1)\rangle,
	\end{equation*}
with $\shA_0 \supset \shA_1 \supset \cdots \supset \shA_{m-1}$ a descending sequence of admissible subcategories, where $\shA_{*}(k) = \shA_{*} \otimes \sO_{\PP(V)}(k)$ denotes the image of $\shA_{*}$ under the autoequivalence $\otimes \sO_{\PP(V)}(k)$ for $k \in \ZZ$. Dually, a {\em left Lefschetz decomposition} of $D(X)$ is a SOD of the form:
	\begin{equation*}\label{duallef:A} 
	D(X) = \langle \shA_{1-m} (1-m), \ldots, \shA_{-1}(-1) , \shA_{0}\rangle,
	\end{equation*}	
	with $\shA_{1-m} \subset \cdots \subset \shA_{-1} \subset \shA_{0}$ an ascending sequence of admissible subcategories. 
	
	The variety $X \to \PP(V)$ is said to be a {\em Lefschetz variety}, or to admit a {\em Lefschetz structure} if $D(X)$ admits both right and left Lefschetz decompositions (with same $\shA_0$ and $m$) as above. If $X$ is a smooth $S$-scheme, then $X$ is a Lefschetz variety if it admits either a right or a left Lefschetz decomposition. The number $m$ is called the {\em length} of the Lefschetz structure. See \cite{Kuz07HPD, Kuz08Lef, JLX17, P18, JL18Bl} for more about Lefschetz decompositions.

Let $Q  = \{(x,[H]) \mid x \in H \}\subset \PP(V) \times_S \PP(V^\vee)$ be the universal quadric for $\PP(V)$ (or equivalently for $\PP(V^\vee)$). Then the {\em universal hyperplane $\shH_{X}$ for $X \to \PP(V)$} is defined to be
	$$\shH_{X}  : = X \times_{\PP(V)} Q \subset X \times_S \PP(V^\vee).$$
Denote $\iota_{\shH}\colon \shH_{X} \to X \times_S \PP(V^\vee)$ the inclusion, then it is easy to show there is a $\PP(V^\vee)$-linear semiorthogonal decomposition (see \cite{Kuz07HPD, RT15HPD, JLX17}):
	$$D(\shH_{X}) = \big \langle \sC, ~~\iota_{\shH}^* (\shA_1(1) \boxtimes_S D(\PP(V^\vee))), \ldots,  \iota_{\shH}^* (\shA_{m-1}((m-1)) \boxtimes_S D(\PP(V^\vee))) \big \rangle.$$

\begin{definition} \label{def:HPD} The category $\sC$ is called the {\em HPD category} of $D(X)$, denoted by $D(X)^\hpd$. If there exists a variety $Y$ with $Y \to \PP(V^\vee)$, and a Fourier-Mukai kernel $\shP \in D(Y \times_{\PP(V^\vee)} \shH_X)$ such that the $\PP(V^\vee)$-linear Fourier Mukai functor $\Phi_{\shP}^{Y \to \shH_X} \colon D(Y) \to D(\shH)$ induces an equivalence of categories $D(Y) \simeq D(X)^\hpd$, then $Y \to \PP(V^\vee)$ is called the {\em homological projective dual variety} or {\em HPD variety} of $X \to \PP(V)$.
\end{definition}

The HPD is a {\em reflexive} relation: $(X^\hpd)^\hpd \simeq X$, see \cite{Kuz07HPD,JLX17}. The primary output of the HPD theory is the Kuznetsov's fundamental theorem of HPD for linear sections \cite{Kuz07HPD}; we refer the readers to the references \cite{Kuz07HPD, Kuz14sod, RT15HPD, JLX17, JL18Bl} for the precise statement of the theorem and its various applications.

\begin{remark} The HPD theory can be set up in the noncommutative setting for a $\PP(V)$--linear Lefschetz category $\shA$, which is a $\PP(V)$--linear category (with proper enhancement) together with a right and left Lefschetz decomposition as above, see \cite{P18, JL18Bl}. Then one can similarly define the HPD category $\shA^\hpd$ of $\shA$, and the fundamental theorem of HPD still holds for dual linear sections of $\shA$ and $\shA^\hpd$, see \cite{JLX17, Ren, P18}.
\end{remark}

\newpage
\section{Generalized linear duality}

As in the introduction, let $V$ and $K$ be vector bundles over $S$ of rank $N \ge 2$ and $k \le N$ respectively, $\sigma \in \Hom_S(K,V)$ be an injective $\sO_S$-module morphism and $\sL = {\rm coker}(\sigma)$ be the cokernel. Denote $\ell = \rank \sL$, therefore $k = N -\ell$. There is a short exact sequence:
	$$0 \to K \to V \to \sL \to 0,$$ 
and the dual sheaves fit into a four-term exact sequence given by (\ref{eqn:dualsL}). Further denote
	$$\tX : = \PP(\sExt^1(\sL,\sO_S)) \subset \PP(K^\vee),$$
which is a desingularization of the degeneracy locus $S_\sigma = \Sing(\sL) \subset S$, and denote by
	$$\widetilde{\PP}(K^\vee): = \Bl_Z \PP(K^\vee) \to \PP(V^\vee)$$
the blowing up of $\PP(K^\vee)$ along $Z \subset  \PP(K^\vee)$. 

\begin{lemma}[Lefschetz decomposition]\label{lem:lef:Bl} The blowing up $\widetilde{\PP}(K^\vee) \to \PP(V^\vee)$ admits a $S$-linear Lefschetz decomposition with respect to the action of $\sO_{\PP(V^\vee)}(1)$: 
 \begin{equation}\label{lef:Bl}
		D(\widetilde{\PP}(K^\vee)) = \big\langle \shA_0, \shA_{1}(1) \ldots, \shA_{r-2}(r-2)\big\rangle, \qquad r = \max\{N, k+1\},
\end{equation}
where $(i)$ denotes the twist by $\sO_{\PP(V^\vee)}(i)$, $i \in \ZZ$, $\shA_0 \supset \shA_1 \supset \ldots \supset \shA_{r-2}$ are given by:			
\begin{align*} 
	& \shA_0 = \ldots  = \shA_{k-1} =  \langle \LL \pi^* \, D(S), D(\tX)_0 \rangle, \quad \shA_{k} =\ldots =  \shA_{N-2} = D(\tX)_0,  & \text{if}~ k \le N-1, \\
	& \shA_0 = \ldots  = \shA_{k-2} =  \langle \LL \pi^*\, D(S), D(\tX)_0 \rangle, \quad \shA_{N-1} =   \LL \pi^*\, D(S),  &\text{if}~ k = N. 
\end{align*}
where $D(\tX)_0$ is the image of $D(\tX)$ under fully faithful embedding $\RR j_*\,\LL p^*$.
\end{lemma}

\begin{proof}
This follows directly from performing right mutations Lem. \ref{lem:mut:Bl} to Orlov's formula Thm. \ref{thm:blow-up} for the blowing up $\widetilde{\PP}(K^\vee)$. (Cf. \cite[Prop. 3.1]{CT15}, \cite[Prop. 4.4]{JL18Bl}).
\end{proof}

Our main result is the following generalization of linear duality:

\begin{theorem}[Generalized linear duality] \label{thm:duality} The $S$-linear scheme $\PP(\sL) \hookrightarrow \PP(V)$ is homological projective dual to $\widetilde{\PP}(K^\vee) \to \PP(V^\vee)$ with respect to the Lefschetz decomposition (\ref{lef:Bl}). 
\end{theorem}


The HPD relation between $\widetilde{\PP}(K^\vee) \to \PP(V^\vee)$ and $\PP(\sL) \hookrightarrow \PP(V)$ of the theorem can be visualized in the following diagram using Kuznetsov's convention \cite{Kuz07HPD}:
\begin{figure}[h]
\centering
\includegraphics[width=0.8  \textwidth]{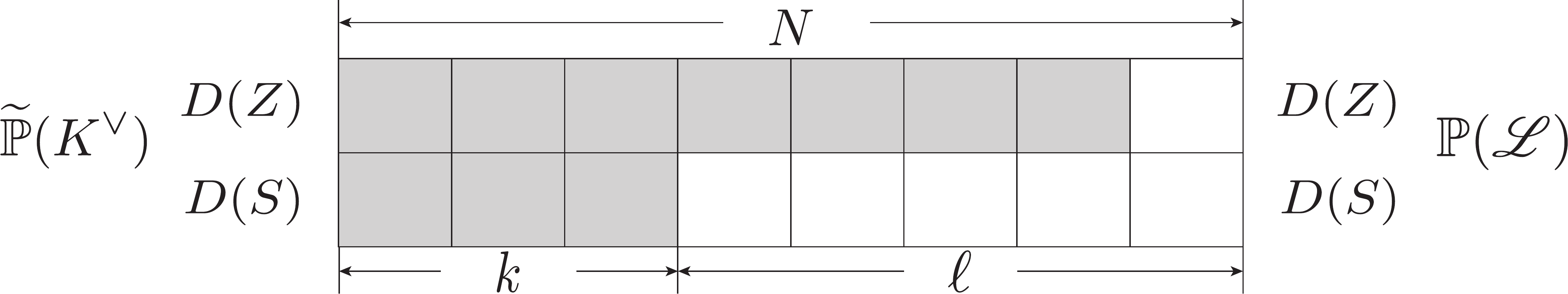}
\end{figure}

The SOD for $D(\PP(\sL))$ obtained from HPD theory applied to Lem. \ref{lem:lef:Bl} agrees with the projectivization formula of \cite[Thm. 3.1]{JL17} (up to mutations). Therefore the above theorem shows the duality between the projectivization formula of \cite{JL17} and the blowing up formula of \cite{Orlov92}, and one can deduce one formula from the other based on results of ``chess game" \cite{JLX17}.

If $\sL$ is locally free, then the above theorem reduces to the usual linear duality" $\PP(L)^\hpd = \PP(L)^\perp \equiv \PP(K^\vee)$. If $K = \sO_S$, and $\sigma=s: \sO \to V$ a regular section, then this is the HP duality between generalized universal hyperplane $\shH_s \subset \PP(V)$ and blowing up $\Bl_Z S \subset \PP(V^\vee)$ (cf. \cite[Prop. 3.2]{CT15}), which can be visualised using Kuznetsov's convention as
	\begin{align*}
	\Bl_Z S = \PP(\sI_Z) \quad \begin{ytableau}
	\none[\ \ D(Z)]
	&\none
	&&&&&&&&&&*(lightgray)
	&\none[\ \ \ \ \ \ D(Z)]
	\\
	\none[\ \ D(S)]
	&\none
&&*(lightgray)&*(lightgray)&*(lightgray)&*(lightgray)&*(lightgray)&*(lightgray)&*(lightgray)&*(lightgray)&*(lightgray)
	&\none[\ \ \ \ \ \  D(S)  ]
	\end{ytableau} \qquad \quad \shH_s = \PP(\sL).
	\end{align*}
If $N =k$, then theorem implies $S_\sigma^+ = \PP(\sL) \subset \PP(V)$ is homological projective dual to $\Bl_{S_\sigma^-} \PP(K^\vee) \to \PP(V^\vee)$, the blowing up along a (in general) {different} resolution $S_\sigma^-=Z$ of singularities of the degeneracy locus $S_\sigma$.  If $k = N-1$, then $S_{\sigma} \subset S$ is a Cohen-Macaulay subscheme of codimension $2$, and $\PP(\sL)=\Bl_{S_\sigma} S$ is the blowing up along $S_{\sigma}$. The theorem states the HPD between the two blowing--ups $\Bl_{S_\sigma} S \subset \PP(V)$ and $\Bl_Z \PP(K^\vee)  \to \PP(V^\vee)$.

\subsubsection*{Proof of Thm. \ref{thm:duality} } 
The situation can be regarded as a relative situation of \cite{CT15, JL18Bl}, and a similar strategy can be applied. Apply the construction of \S \ref{sec:Cayley} to the scheme $X=\PP(K^\vee)$ and the zero locus  $i: \tX \hookrightarrow \PP(K)$ of the canonical regular section of vector bundle $\sE = V \boxtimes \sO_{\PP(K^\vee)}(1)$, then the generalized universal hyperplane $\iota \colon \shH: = \shH_s \subset \PP(V) \times_S \PP(K^\vee)$ is a divisor of the line bundle $\sO_{\PP(V)}(1) \boxtimes \sO_{\PP(K^\vee)}(1)$. Consider the blowing up $\beta \colon \widetilde{\PP}(K^\vee) \to \PP(K^\vee)$. Its exceptional divisor is given by $\PP(\sN_{i}^\vee) =\PP(V^\vee) \times_S \tX$. Apply the geometry of  \S \ref{sec:BlvsCayley}, we get that the blowing up  $\gamma: \sU \to \shH$ of $\shH$ along $j \colon \PP(V) \times_S Z \hookrightarrow \shH$ is the universal hyperplane for $\widetilde{\PP}(K^\vee) \to \PP(V)$, i.e. $\iota_{\sU} :  \sU = \shH_{\widetilde{\PP}(K^\vee)} \hookrightarrow  \widetilde{\PP}(K^\vee)  \times_S \PP(V),$
and the exceptional locus of $\gamma$ is $j_Q \colon Q_Z \hookrightarrow \sU$, where $Q_Z \subset \PP(V) \times_S \PP(V^\vee) \times_S \tX$ is the base-change of the universal quadric $Q \subset  \PP(V) \times_S \PP(V^\vee)$ along map $Z \to S$. The situation is summarized in the following diagram, with notation of maps as indicated:

\begin{equation} \label{diagram:hilb}
\begin{tikzcd}[back line/.style={}, column sep = 1 em]
& \PP(V) \times_S \tX  \ar[back line]{dd}[near start,swap]{p} \ar[hook]{rr}{j}
  & & \shH \ar{dd}[near start]{\pi_{\shH}} \ar[hook]{rr}{\iota} & & \PP(V) \times_S \PP(K^\vee) \ar[bend left]{lldd}{q}
  \\
Q_Z \ar{dd}[swap]{\check{\pi}_Q} \ar{ur}{\pi_Q}
  & &  \sU  \ar{ur}{\gamma}  \ar[crossing over, <-right hook]{ll}[near start, swap]{j_Q}  \ar[crossing over, hook]{rr}{\iota_\sU~~} & & \PP(V) \times_S \widetilde{\PP}(K^\vee) \ar{ur}{\widetilde{\beta}} \ar[crossing over,bend left]{lldd}{\widetilde{q}}
  \\
& \tX \ar[back line, hook]{rr}[near start]{i}
  & & \PP(K^\vee)
   \\
\PP(V^\vee) \times_S \tX \ar{ur}{\check{p}} \ar[hook]{rr}{\check{j}} & & \widetilde{\PP}(K^\vee) \ar[crossing over]{ur}{\beta} \ar[crossing over, leftarrow]{uu}[near end]{\pi_\sU} 
\end{tikzcd}
\end{equation}
\medskip

In the rest of the proof we will write derived functors as {\em underived}, for simplicity of notations. From blowing up formula for $\gamma:\sU \to \shH$, we have 
	\begin{equation}\label{sod:sU'}
	D(\sU) = \big \langle \gamma^* D(\shH), ~ D(\PP(V) \times_S \tX)_0, \ldots, D(\PP(V) \times_S \tX)_{N-3} \big \rangle.
	\end{equation}
where $D(\PP(V) \times_S \tX)_{k} = j_{Q*}\, \pi_Q^* \, D(\PP(V) \times_S \tX) \otimes \sO(-k E_Q)$ (where $Q_Z$ is the exceptional divisor and $E_Q$ denotes the divisor class of $Q_Z$).  It follows directly from $\sO(-E_Q) = \sO_{\PP(V^\vee)}(1) \otimes \sO_{\PP(K^\vee)}(-1)$ and the diagram that
	$$D(\PP(V) \times_S \tX)_k =  D(\tX)_k \boxtimes_S \PP(V)|_\sU.$$ 
On the other hand, as observed in \cite{JL17}, $\shH$ is also the generalized universal hyperplane for the scheme $X_1 = \PP(V)$ and the zero locus $i_1 \colon \PP(\sL) \hookrightarrow \PP(V)$ of a canonical section of the vector bundle $\sE_1 = K^\vee \otimes \sO_{\PP(V)}(1)$. Denote $\pi_1 \colon \shH \to \PP(V)$ the projection, $j_1 \colon \PP_{\PP(\sL)} (\sN_{j_1}) \hookrightarrow \shH$ the inclusion and $p_1: \PP_{\PP(\sL)} (\sN_{j_1})  \to \PP(\sL)$ the projection. Then by Thm. \ref{thm:HPDI},
	\begin{equation}\label{sod:H'}
	D(\shH)  = \big \langle \Phi_1 ( D(\PP(\sC_\sigma)), ~ \pi_1^*D(\PP(V)) \otimes \sO_{\PP(K^\vee)}, \ldots, \pi_1^* D(\PP(V)) \otimes \sO_{\PP(K^\vee)}(k-2) \big \rangle,\\
	\end{equation}
where $\Phi_1 = j_{1*} \, p_1^*(-) \otimes \sO_{\PP(K^\vee)}(-1)$. From diagram (\ref{diagram:hilb}) we have
	\begin{equation}\label{eqn:mutated}\gamma^* ( \pi_1^* D(\PP(V)) \otimes \sO_{\PP(K^\vee)}(k)) = (\pi^* D(S) \otimes \sO_{\PP(K^\vee)} (k))\boxtimes_S\PP(V)|_\sU.\end{equation}
By Lem. \ref{lem:mut:Bl}, each time one right mutates (\ref{eqn:mutated}) passing through some $D(\PP(V) \times_S \tX)_{k'}$ inside (\ref{sod:sU'}) will result in tensoring (\ref{eqn:mutated}) with $\sO(-E_Q)$ and thus gets
	$$\big(\pi^* D(S) \otimes \sO_{\PP(K^\vee)} (k-1) \otimes  \sO_{\PP(V^\vee)}(1)\big)\boxtimes_S D(\PP(V))|_\sU.$$
Repeating this process of mutations inside (\ref{sod:sU'}) and substitute the category $D(\shH)$ by (\ref{sod:H'}), we end up with the following SOD:
	\begin{equation*} \label{sod:sU}
	D(\sU) = \big \langle \Psi \, ( D(\PP(\sC_\sigma))) , (\shA_1(1) \boxtimes_S D(\PP(V)) )|_\sU, \ldots, (\shA_{r-2}(r-2) \boxtimes_S D(\PP(V)) )|_\sU \big \rangle,
	\end{equation*}
where $\shA_i$'s are given by Lem. \ref{lem:lef:Bl}, and $\Psi \, = \LL \gamma^* \RR j_{1*} \LL p_1^*(-) \otimes \sO_{\PP(V^\vee)}(1) \otimes \sO_{\PP(K^\vee)}(-1): D(\PP(\sL)) \hookrightarrow D(\sU)$. By definition of HPD (Def. \ref{def:HPD}), we are done. \hfill $\square$			

\subsubsection*{Proof of Thm. \ref{thm:HPD}}  Apply the categorical Pl\"ucker formula of \cite{JLX17} to the two HPD pairs $(\shA/\PP(V^\vee), \shA^\hpd/\PP(V))$ and $(\PP(\sL) \subset \PP(V^\vee), \PP(\sL)^\perp:=\widetilde{\PP}(K^\vee) \to \PP(V))$, then the theorem  \ref{thm:HPD} immediately follows. \footnote{Notice that one could also apply the nonlinear HPD theorem of \cite{KP18,JL18join} to our theorem \ref{thm:duality} and obtain similar results in a slightly different formulation.}
\hfill $\square$			


\newpage
\begin{thebibliography}{99}

\bibitem[AW]{AW93}
Andreatta, M., and Wi{\'s}niewski, J. A.,
{\em A note on nonvanishing and applications.} 
Duke Mathematical Journal 72.3 (1993): 739-755.

\bibitem[BK]{BK}
Bondal, A., Kapranov, M.,
{\em Representable functors, Serre functors, and reconstructions,}
(Russian) Izv. Akad. Nauk SSSR Ser. Mat. 53 (1989), no. 6, 1183-1205, 1337; translation in Math. USSR-Izv. 35 (1990), no. 3, 519-541.

\bibitem[C{\u{a}}l]{Cal}
C{\u{a}}ld{\u{a}}raru, A.
{\em Derived categories of sheaves: a skimming}  Snowbird Lectures in Algebraic Geometry. Contemp. Math 388 (2005): 43-75.

\bibitem[CT15]{CT15}
Carocci, F., Turcinovic Z.. 
{\em Homological projective duality for linear systems with base locus.} 
arXiv preprint arXiv:1511.09398 (2015).

\bibitem[Huy]{Huy}
Huybrechts, D.,
{\em Fourier-Mukai transforms in algebraic geometry}.
Oxford University Press (2006).

\bibitem[JLX17]{JLX17}
Jiang, Q., Leung, N.C., and Xie, Y., 
{\em Categorical Pl\" ucker formula and homological projective duality.} 
arXiv \href{https://arxiv.org/abs/1704.01050}{1704.01050} (2017).

\bibitem[JL18a]{JL18Bl}
Jiang, Q., Leung, N.C., 
{\em Blowing up linear categories, refinements, and homological projective duality with base locus.} arXiv: \href{https://arxiv.org/abs/1811.05132}{1811.05132}  (2018).

\bibitem[JL18b]{JL18join}
Jiang, Q., Leung, N.C.,
 {\em Categorical duality between joins and intersections.} arXiv: \href{https://arxiv.org/abs/1811.05135}{1811.05135} (2018).

\bibitem[JL18c]{JL17}
Jiang, Q., Leung, N.C.,
 {\em Derived category of projectivization and flops.} arXiv:\href{https://arxiv.org/abs/1811.12525}{1811.12525} (2018).

\bibitem[K07]{Kuz07HPD} 
Kuznetsov, A., 
{\em Homological projective duality},
Publications Mathematiques de L'IHES, {\bf 105}, n. 1 (2007), 157-220.

\bibitem[K08]{Kuz08Lef}
Kuznetsov, A.,
{\em Lefschetz decompositions and categorical resolutions of singularities.} Selecta Mathematica, New Series 13.4 (2008): 661-696.

\bibitem[K11]{Kuz11Bas} 
Kuznetsov, A., 
{\em  Base change for semiorthogonal decompositions.} 
Comp.\ Math.\ 147 (2011), 852--876.

\bibitem[K14]{Kuz14sod}
Kuznetsov, A., 
{\em Semiorthogonal decompositions in algebraic geometry},
preprint math.AG/1404.3143, Proceedings of ICM-2014.

\bibitem[KP18]{KP18} Kuznetsov, A., and Perry, A.,
 {\em Categorical joins.} arXiv:1804.00144 (2018).

\bibitem[O92]{Orlov92}
Orlov, D.,
{\em Projective bundles, monoidal transformations, and derived categories of coherent sheaves}. Izvestiya Rossiiskoi Akademii Nauk. Seriya Matematicheskaya, 56(4). 852-862. Chicago (1992).

\bibitem[O05]{Orlov05}
Orlov, D.,
{\em Triangulated categories of singularities and equivalences between Landau-Ginzburg models.} 
Sbornik: Mathematics 197.12 (2006): 1827.00-224.

\bibitem[P18]{P18}
Perry, A.
{\em Noncommutative homological projective duality.} arXiv:1804.00132 (2018).

\bibitem[R17]{Ren}
Rennemo, J. V.
{\em The fundamental theorem of homological projective duality via variation of GIT stability}. arXiv:1705.01437 (2017).

\bibitem[T15]{RT15HPD}
Thomas, R., 
{\em Notes on HPD}.
preprint math.AG/1512.08985v3.


\end{thebibliography}
\end{document}